\documentclass[12pt,a4paper,twoside]{article}

\title{\sc Maxwell meets Korn:\\ 
A New Coercive Inequality\\
for Tensor Fields in $\mathbb{R}^{N\times N}$
with Square-Integrable Exterior Derivative}
\def\shorttitle{Maxwell meets Korn}
\def\pauthor{Patrizio Neff, Dirk Pauly, Karl-Josef Witsch}

\def\mylabelonoff{off}
\def\allowdisbrk{no}

\usepackage{a4,exscale,ifthen,amsfonts,amssymb,amsmath,amscd,graphicx}
\usepackage[english]{babel}
\usepackage[mathscr]{eucal}

\author{{\sf\pauthor}}
\numberwithin{equation}{section}

\newenvironment{acknow}{{\vspace*{1cm}\noindent\bf Acknowledgements }}{}
\newcommand{\bewboxw}{\mbox{}\hfill $\square$ \\}

\newenvironment{proof}{{\noindent\bf Proof }}{\bewboxw}

\newcommand{\keywords}[1]{{\noindent\bf Key Words }#1}

\ifthenelse{\equal{\mylabelonoff}{on}}
{}{}
\ifthenelse{\equal{\allowdisbrk}{yes}}
{\allowdisplaybreaks}{}

\usepackage[mathscr]{eucal}
\usepackage{color}

\newcommand{\ol}{\overline}

\newcommand{\nz}{\mathbb{N}}

\newcommand{\rz}{\mathbb{R}}

\newcommand{\rN}{\rz^N}

\DeclareMathOperator{\p}{\partial}

\newcommand{\na}{\nabla}

\DeclareMathOperator{\ed}{d}

\DeclareMathOperator{\cd}{\delta}

\DeclareMathOperator{\grad}{grad}
\DeclareMathOperator{\Grad}{Grad}

\DeclareMathOperator{\curl}{curl}
\DeclareMathOperator{\Curl}{Curl}

\renewcommand{\div}{\operatorname{div}}
\DeclareMathOperator{\Div}{Div}

\newcommand{\trans}[1]{{#1}^t}
\newcommand{\T}{T}
\newcommand{\TS}{S}

\newcommand{\om}{\Omega}
\newcommand{\dom}{\p\!\om}
\newcommand{\ga}{\Gamma}

\newcommand{\ttrga}{\tau_{\ga}}

\DeclareMathOperator{\sym}{sym}

\renewcommand{\skew}{\operatorname{skew}}

\newcommand{\dvec}[3]{\begin{bmatrix}#1\\#2\\#3\end{bmatrix}}

\def\set#1#2{\{#1\,:\,#2\}}

\DeclareMathOperator{\Lebesgue}{\mathsf{L}}
\newcommand{\Lgen}[2]{\Lebesgue^{#1}_{#2}}

\def\Lt{\Lgen{2}{}}
\def\Ltom{\Lt(\om)}

\newcommand{\qLt}[1]{\Lgen{2,#1}{}}
\newcommand{\Ltq}{\qLt{q}}
\newcommand{\Ltqpo}{\qLt{q+1}}
\newcommand{\Ltqmo}{\qLt{q-1}}
\newcommand{\Ltqom}{\Ltq(\om)}
\newcommand{\Ltqpoom}{\Ltqpo(\om)}
\newcommand{\Ltqmoom}{\Ltqmo(\om)}

\DeclareMathOperator{\DSobolev}{\mathsf{D}}
\newcommand{\Dgen}[3]{\overset{#3}{\DSobolev}{}^{#1}_{#2}}
\newcommand{\qD}[1]{\Dgen{#1}{}{}}

\newcommand{\qDc}[1]{\Dgen{#1}{}{\circ}}
\newcommand{\qDcz}[1]{\Dgen{#1}{0}{\circ}}
\newcommand{\Dq}{\qD{q}}

\newcommand{\Dqc}{\qDc{q}}

\newcommand{\Dqmoc}{\qDc{q-1}}
\newcommand{\Dqcz}{\qDcz{q}}

\newcommand{\Dqom}{\Dq(\om)}

\newcommand{\Dqcom}{\Dqc(\om)}

\newcommand{\Dqmocom}{\Dqmoc(\om)}
\newcommand{\Dqczom}{\Dqcz(\om)}

\DeclareMathOperator{\DeSobolev}{\Delta}
\newcommand{\Degen}[3]{\overset{#3}{\DeSobolev}{}^{#1}_{#2}}
\newcommand{\qDe}[1]{\Degen{#1}{}{}}

\newcommand{\qDez}[1]{\Degen{#1}{0}{}}

\newcommand{\Deq}{\qDe{q}}

\newcommand{\Deqz}{\qDez{q}}

\newcommand{\Deqmoz}{\qDez{q-1}}

\newcommand{\Deqom}{\Deq(\om)}

\newcommand{\Deqzom}{\Deqz(\om)}

\newcommand{\Deqmozom}{\Deqmoz(\om)}

\DeclareMathOperator{\Sobolev}{\mathsf{H}}
\newcommand{\Hgen}[3]{\overset{#3}{\Sobolev}{}^{#1}_{#2}}
\def\Ho{\Hgen{1}{}{}}

\def\Hoz{\Hgen{1}{}{\circ}}
\def\Hoom{\Ho(\om)}
\def\Hozom{\Hoz(\om)}

\DeclareMathOperator{\Cont}{\mathsf{C}}
\newcommand{\Cgen}[2]{\overset{#2}{\Cont}{}^{#1}}

\def\Cic{\Cgen{\infty}{\circ}}

\def\Ciqc{\Cgen{\infty,q}{\circ}}

\def\Ciqpoc{\Cgen{\infty,q+1}{\circ}}

\def\Cicom{\Cic(\om)}

\def\Ciqcom{\Ciqc(\om)}

\def\Ciqpocom{\Ciqpoc(\om)}

\DeclareMathOperator{\dirichlet}{\mathcal{H}}
\newcommand{\qharmdi}[2]{\dirichlet^{#1}_{#2}(\om)}

\newcommand{\harmdiz}{\qharmdi{0}{}}
\newcommand{\harmdio}{\qharmdi{1}{}}
\newcommand{\harmdiq}{\qharmdi{q}{}}

\newcommand{\harmdiqmo}{\qharmdi{q-1}{}}

\newcommand{\Hggen}[3]{\overset{#2}{\Sobolev}(\grad;#3)}
\newcommand{\HGgen}[3]{\overset{#2}{\Sobolev}(\Grad;#3)}
\newcommand{\Hcgen}[3]{\overset{#2}{\Sobolev}(\curl_{#1};#3)}
\newcommand{\HCgen}[3]{\overset{#2}{\Sobolev}(\Curl_{#1};#3)}
\newcommand{\Hdgen}[3]{\overset{#2}{\Sobolev}(\div_{#1};#3)}
\newcommand{\HDgen}[3]{\overset{#2}{\Sobolev}(\Div_{#1};#3)}

\newcommand{\Hgom}{\Hggen{}{}{\om}}
\newcommand{\HGom}{\HGgen{}{}{\om}}
\newcommand{\Hcom}{\Hcgen{}{}{\om}}
\newcommand{\HCom}{\HCgen{}{}{\om}}
\newcommand{\Hdom}{\Hdgen{}{}{\om}}
\newcommand{\HDom}{\HDgen{}{}{\om}}
\newcommand{\Hgcom}{\Hggen{}{\circ}{\om}}
\newcommand{\HGcom}{\HGgen{}{\circ}{\om}}
\newcommand{\Hccom}{\Hcgen{}{\circ}{\om}}
\newcommand{\HCcom}{\HCgen{}{\circ}{\om}}

\newcommand{\Hczom}{\Hcgen{0}{}{\om}}
\newcommand{\HCzom}{\HCgen{0}{}{\om}}
\newcommand{\Hdzom}{\Hdgen{0}{}{\om}}
\newcommand{\HDzom}{\HDgen{0}{}{\om}}
\newcommand{\Hcczom}{\Hcgen{0}{\circ}{\om}}
\newcommand{\HCczom}{\HCgen{0}{\circ}{\om}}

\newcommand{\normdst}{\hspace{-0.4ex}}
\newcommand{\scp}[2]{\left\langle#1,#2\right\rangle}

\newcommand{\scpLtom}[2]{\scp{#1}{#2}_{\Ltom}}
\newcommand{\scpLtqom}[2]{\scp{#1}{#2}_{\Ltqom}}
\newcommand{\scpLtqpoom}[2]{\scp{#1}{#2}_{\Ltqpoom}}

\newcommand{\norm}[1]{\left|\normdst\left|#1\right|\normdst\right|}
\newcommand{\dnorm}[1]{\left|\normdst\left|\normdst\left|#1\right|\normdst\right|\normdst\right|}

\newcommand{\normLtom}[1]{\norm{#1}_{\Ltom}}
\newcommand{\normLtqom}[1]{\norm{#1}_{\Ltqom}}
\newcommand{\normLtqpoom}[1]{\norm{#1}_{\Ltqpoom}}
\newcommand{\normLtqmoom}[1]{\norm{#1}_{\Ltqmoom}}

\newtheorem{lem}{Lemma}

\newtheorem{theo}[lem]{Theorem}
\newtheorem{cor}[lem]{Corollary}
\newtheorem{rem}[lem]{Remark}

\renewcommand{\T}{P}
\renewcommand{\trans}[1]{{#1}^T}

\newcommand{\soN}{\mathfrak{so}(N)}
\newcommand{\ped}{\ed\!}
\newcommand{\pcd}{\cd\!}
\newcommand{\pp}{\p\!}

\begin{document}

\maketitle{}

\begin{abstract}
\noindent
For a bounded domain $\om\subset\rN$
with connected Lipschitz boundary we prove
the existence of some $c>0$, such that 
$$c\norm{\T}_{\Lt(\om,\rz^{N\times N})}
\leq\norm{\sym\T}_{\Lt(\om,\rz^{N\times N})}
+\norm{\Curl\T}_{\Lt(\om,\rz^{N\times(N-1)N/2})}$$
holds for all square-integrable tensor fields $\T:\om\to\rz^{N\times N}$,
having square-integrable generalized `rotation' 
$\Curl\T:\om\to\rz^{N\times(N-1)N/2}$
and vanishing tangential trace on $\dom$,
where both operations are to be understood row-wise.
Here, in each row the operator $\curl$
is the vector analytical reincarnation
of the exterior derivative $\ped$ in $\rN$.
For compatible tensor fields $\T$, i.e., $\T=\na v$,
the latter estimate reduces to a non-standard variant
of Korn's first inequality in $\rN$, namely
$$c\norm{\na v}_{\Lt(\om,\rz^{N\times N})}
\leq\norm{\sym\na v}_{\Lt(\om,\rz^{N\times N})}$$
for all vector fields $v\in\Ho(\om,\rN)$, 
for which $\na v_{n}$, $n=1,\dots,N$, are normal at $\dom$.\\
\end{abstract}


\section{Introduction and Preliminaries}

We extend the results from \cite{neffpaulywitschgenkornrt},
which have been announced in \cite{neffpaulywitschgenkornpamm}, 
to the $N$-dimensional case following
in close lines the arguments presented there.
Let $N\in\nz$ and $\om$ be a bounded domain in $\rN$
with connected Lipschitz boundary $\ga:=\dom$.
We prove a Korn-type inequality in $\HCcom$
for eventually non-symmetric tensor fields $\T$
mapping $\om$ to $\rz^{N\times N}$.
More precisely, there exists a positive constant $c$, such that 
$$c\normLtom{\T}\leq\normLtom{\sym\T}+\normLtom{\Curl\T}$$
holds for all tensor fields $\T\in\HCcom$,
where $\T$ belongs to $\HCcom$, if
$\T\in\HCom$ has vanishing tangential trace on $\ga$.
Thereby, the generalized $\Curl$ and tangential trace are defined
as row-wise operations. 
For compatible tensor fields $\T=\na v$ with vector fields $v\in\Hoom$,
for which $\na v_{n}$, $n=1,\dots,N$, are normal at $\dom$,
the latter estimate reduces to a non-standard variant
of the well known Korn's first inequality in $\rN$
$$c\normLtom{\na v}\leq\normLtom{\sym\na v}.$$
Our proof relies on three essential tools, namely
\begin{enumerate}
\item Maxwell estimate (Poincar\'e-type estimate),
\item Helmholtz' decomposition,
\item Korn's first inequality.
\end{enumerate}
In \cite{neffpaulywitschgenkornrt} we already pointed out
the importance of the Maxwell estimate and the related question
of the Maxwell compactness property\footnote{By `Maxwell estimate' 
and `Maxwell compactness property' we mean the estimates 
and compact embedding results used in the theory of Maxwell's equations.}.
Here, we mention the papers
\cite{costabelremmaxlip,kuhndiss,picardpotential,picardcomimb,picarddeco,picardweckwitschxmas,weckmax}.
Results for the Helmholtz decomposition can be found in
\cite{friedrichsdiffformsriemann,paulydeco,picardpotential,picarddeco,weckmax,sproessighelmdecooverview,
mitreadorinamariusfinensolhodgedeco,mitreadorinamariusshawtracediffformhodgedeco,mitreamariussharphodgedecomax}.
Nowadays, differential forms find prominent applications in numerical methods like
Finite Element Exterior Calculus
\cite{arnoldfalkwintherfemec,hiptmairfemmax}
or Discrete Exterior Calculus
\cite{hiranidiss}.

\subsection{Differential Forms}

We may look at $\om$ as a smooth Riemannian manifold of dimension $N$ 
with compact closure and connected Lipschitz continuous boundary $\ga$.
The alternating differential forms of rank $q\in\{0,\dots,N\}$ on $\om$,
briefly $q$-forms, with square-integrable coefficients will be denoted by
$\Ltqom$. The exterior derivative $\ped$ 
and the co-derivative $\pcd=\pm*\ped*$ ($*$: Hodge's star operator)
are formally skew-adjoint to each other, i.e.,
$$\forall\,E\in\Ciqcom\quad H\in\Ciqpocom\qquad
\scpLtqpoom{\ped E}{H}=-\scpLtqom{E}{\pcd H},$$
where the $\Ltqom$-scalar product is given by
$$\forall\,E,H\in\Ltqom\qquad
\scpLtqom{E}{H}:=\int_{\om}E\wedge*H.$$
Here $\Ciqcom$ denotes the space of compactly supported 
and smooth $q$-forms on $\om$.
Using this duality, we can define weak versions of $\ped$ and $\pcd$.
The corresponding standard Sobolev spaces are denoted by
\begin{align*}
\Dqom&:=\set{E\in\Ltqom}{\ped E\in\Ltqpoom},\\
\Deqom&:=\set{H\in\Ltqom}{\pcd H\in\Ltqmoom}.
\end{align*}
The homogeneous tangential boundary condition $\ttrga E=0$,
where $\ttrga$ denotes the tangential trace,
is generalized in the space 
$$\Dqcom:=\ol{\Ciqcom},$$
where the closure is taken in $\Dqom$.
In classical terms, we have for smooth $q$-forms
$\ttrga=\iota^{*}$ with the canonical embedding
$\iota:\ga\hookrightarrow\ol{\om}$.
An index $0$ at the lower right position 
indicates vanishing derivatives, i.e.,
\begin{align*}
\Dqczom&=\set{E\in\Dqcom}{\ped E=0},&
\Deqzom&=\set{H\in\Deqom}{\pcd H=0}.
\end{align*} 
By definition and density, we have
$$\Deqzom:=(\ped\Dqmocom)^{\bot},\quad
\Deqzom^{\bot}:=\ol{\ped\Dqmocom},$$
where $\bot$ denotes the orthogonal complement
with respect to the $\Ltqom$-scalar product
and the closure is taken in $\Ltqom$.
Hence, we obtain the $\Ltqom$-orthogonal decomposition,
usually called Hodge-Helmholtz decomposition,
\begin{align}
\label{helmtrivial}
\Ltqom=\ol{\ped\Dqmocom}\oplus\Deqzom,
\end{align}
where $\oplus$ denotes the orthogonal sum
with respect to the $\Ltqom$-scalar product.
In \cite{weckmax,picardcomimb} the following crucial tool has been proved:

\begin{lem}
\label{maxcomp}
{\sf(Maxwell Compactness Property)}
For all $q$ the embeddings
$$\Dqcom\cap\Deqom\hookrightarrow\Ltqom$$
are compact.
\end{lem}

As a first immediate consequence, 
the spaces of so called `harmonic Dirichlet forms'
$$\harmdiq:=\Dqczom\cap\Deqzom$$
are finite dimensional.
In classical terms, a $q$-form $E$ belongs to $\harmdiq$, if
$$\ped E=0,\quad\pcd E=0,\quad\iota^{*}E=0.$$
The dimension of $\harmdiq$ equals the $(N-q)$th Betti number of $\om$.
Since we assume the boundary $\ga$ to be connected,
the $(N-1)$th Betti number of $\om$ vanishes 
and therefore there are no Dirichlet forms of rank $1$ besides zero, i.e.,
\begin{align}
\label{harmdiozero}
\harmdio=\{0\}.
\end{align}
This condition on the domain $\om$ resp. its boundary $\ga$ 
is satisfied e.g. for a ball or a torus.

By a usual indirect argument,
we achieve another immediate consequence:

\begin{lem}
\label{poincarediff}
{\sf(Poincar\'e Estimate for Differential Forms)}
For all $q$ there exist positive constants $c_{p,q}$, 
such that for all $E\in\Dqcom\cap\Deqom\cap\harmdiq^{\bot}$
$$\normLtqom{E}\leq 
c_{p,q}\big(\normLtqpoom{\ped E}^2+\normLtqmoom{\pcd E}^2\big)^{1/2}.$$
\end{lem}

Since 
$$\ped\Dqmocom\subset\Dqczom$$ 
(note that $\ped\ped=0$ and $\pcd\pcd=0$ hold even in the weak sense) 
we get by \eqref{helmtrivial}
$$\ped\Dqmocom=\ped\big(\Dqmocom\cap\Deqmozom\big)
=\ped\big(\Dqmocom\cap\Deqmozom\cap\harmdiqmo^{\bot}\big).$$
Now, Lemma \ref{poincarediff} shows that
$\ped\Dqmocom$ is already closed.
Hence, we obtain a refinement of \eqref{helmtrivial}

\begin{lem}
\label{hodgehelmdeco}
{\sf(Hodge-Helmholtz Decomposition for Differential Forms)}
The decomposition
$$\Ltqom=\ped\Dqmocom\oplus\Deqzom$$
holds.
\end{lem}

\subsection{Functions and Vector Fields}

Let us turn to the special case $q=1$.
In this case, 
we choose (e.g.) the identity as single global chart for $\om$ 
and use the canonical identification isomorphism for $1$-forms 
(i.e., Riesz' representation theorem) 
with vector fields $\ped x_{n}\cong e^n$, namely
$$\sum_{n=1}^{N}v_{n}(x)\ped x_{n}\cong v(x)
=\dvec{v_{1}(x)}{\vdots}{v_{N}(x)},\quad x\in\om.$$ 
$0$-forms will be isomorphically identified with functions on $\om$.
Then, $\ped\cong\grad=\na$ for $0$-forms (functions) and
$\pcd\cong\div=\na\,\cdot\,$ for $1$-forms (vector fields).
Hence, the well known first order differential operators
from vector analysis occur.
Moreover, on $1$-forms we define a new operator
$\curl:\cong\ped$, which turns into the usual $\curl$ if $N=3$ or $N=2$.
$\Ltqom$ equals the usual Lebesgue spaces 
of square integrable functions or vector fields on $\om$ 
with values in $\rz^{n}$, $n:=n_{N,q}:=\binom{N}{q}$, 
which will be denoted by $\Ltom:=\Lt(\om,\rz^{n})$.
$\qD{0}(\om)$ and $\qDe{1}(\om)$ are identified with the
standard Sobolev spaces 
\begin{align*}
\Hgom&:=\set{u\in\Lt(\om,\rz)}{\grad u\in\Lt(\om,\rz^{N})}=\Hoom,\\
\Hdom&:=\set{v\in\Lt(\om,\rz^{N})}{\div v\in\Lt(\om,\rz)},
\end{align*}
respectively. Moreover, we may now identify $\qD{1}(\om)$ with
$$\Hcom:=\set{v\in\Lt(\om,\rz^{N})}{\curl v\in\Lt(\om,\rz^{(N-1)N/2})},$$
which is the well known $\Hcom$ for $N=2,3$.
E.g., for $N=4$ we have
$$\curl v
=\begin{bmatrix}
\pp_{1}v_{2}-\pp_{2}v_{1}\\
\pp_{1}v_{3}-\pp_{3}v_{1}\\
\pp_{1}v_{4}-\pp_{4}v_{1}\\
\pp_{2}v_{3}-\pp_{3}v_{2}\\
\pp_{2}v_{4}-\pp_{4}v_{2}\\
\pp_{3}v_{4}-\pp_{4}v_{3}
\end{bmatrix}\in\rz^6$$
and for $N=5$ we get $\curl v\in\rz^{10}$. 
In general, the entries of the $(N-1)N/2$-vector $\curl v$ consist of 
all possible combinations of 
$$\pp_{n}v_{m}-\pp_{m}v_{n},\quad1\leq n<m\leq N.$$
Similarly, we obtain the closed subspaces 
$$\Hgcom=\Hozom,\quad\Hccom$$ 
as reincarnations of 
$\qDc{0}(\om)$ and $\qDc{1}(\om)$, respectively.
We note
$$\Hgcom=\ol{\Cicom},\quad\Hccom=\ol{\Cicom},$$ 
where the closures are taken in the respective graph norms,
and that in these Sobolev spaces the classical homogeneous scalar 
and tangential (compare to $N=3$) boundary conditions
$$u|_{\ga}=0,\quad\nu\times v|_{\ga}=0$$
are generalized.
Here, $\nu$ denotes the outward unit normal for $\ga$.
Furthermore, we have the spaces of irrotational or solenoidal vector fields
\begin{align*}
\Hczom
&=\set{v\in\Hcom}{\curl v=0},\\
\Hcczom
&=\set{v\in\Hccom}{\curl v=0},\\
\Hdzom
&=\set{v\in\Hdom}{\div v=0}.
\end{align*}
Again, all these spaces are Hilbert spaces. 
Now, we have two compact embeddings
$$\Hgcom\hookrightarrow\Ltom,\quad
\Hccom\cap\Hdom\hookrightarrow\Ltom,$$
i.e., Rellich's selection theorem and the Maxwell compactness property.
Moreover, the following Poincar\'e and Maxwell estimates hold:

\begin{cor}
\label{poincare}
{\sf(Poincar\'e Estimate for Functions)}
Let $c_{p}:=c_{p,0}$. 
Then, for all functions $u\in\Hgcom$
$$\normLtom{u}\leq c_{p}\normLtom{\grad u}.$$
\end{cor}

\begin{cor}
\label{poincaremax}
{\sf(Maxwell Estimate for Vector Fields)}
Let $c_{m}:=c_{p,1}$. 
Then, for all vector fields $v\in\Hccom\cap\Hdom$
$$\normLtom{v}\leq 
c_{m}\big(\normLtom{\curl v}^2+\normLtom{\div v}^2\big)^{1/2}.$$
\end{cor}

We note that generally $\harmdiz=\{0\}$ 
and by \eqref{harmdiozero} also $\harmdio=\{0\}$.
The appropriate Helmholtz decomposition for our needs is 

\begin{cor}
\label{helmdeco}
{\sf(Helmholtz Decomposition for Vector Fiels)}
$$\Ltom=\grad\Hgcom\oplus\Hdzom$$
\end{cor}
 
\subsection{Tensor Fields}

We extend our calculus to $(N\times N)$-tensor (matrix) fields.
For vector fields $v$ with components in $\Hgom$
and tensor fields $\T$ with rows in $\Hcom$ resp. $\Hdom$, i.e.,
$$v=\dvec{v_{1}}{\vdots}{v_{N}},\quad 
v_{n}\in\Hgom,\quad
\T=\dvec{\trans{\T_{1}}}{\vdots}{\trans{\T_{N}}},\quad
\T_{n}\in\Hcom\text{ resp. }\Hdom$$
for $n=1,\dots,N$, we define
$$\Grad v:=\dvec{\trans{\grad}v_{1}}{\vdots}{\trans{\grad}v_{N}}=J_{v}=\na v,\quad
\Curl\T:=\dvec{\trans{\curl}\T_{1}}{\vdots}{\trans{\curl}\T_{N}},\quad
\Div\T:=\dvec{\div\T_{1}}{\vdots}{\div\T_{N}},$$ 
where $J_{v}$ denotes the Jacobian of $v$ and $\trans{}$ the transpose.
We note that $v$ and $\Div\T$ are $N$-vector fields,
$\T$ and $\Grad v$ are $(N\times N)$-tensor fields, 
whereas $\Curl\T$ is a $(N\times(N-1)N/2)$-tensor field 
which may also be viewed as a totally anti-symmetric 
third order tensor field with entries
$$(\Curl\T)_{ijk}=\pp_j\T_{ik}-\pp_k\T_{ij}.$$
The corresponding Sobolev spaces will be denoted by
\begin{align*}
&\HGom,&&\HGcom,&&\HDom,&&\HDzom,\\
&\HCom,&&\HCcom,&&\HCzom,&&\HCczom.
\end{align*}

There are three crucial tools to prove our estimate.
First, we have obvious consequences 
from Corollaries \ref{poincare}, \ref{poincaremax} and \ref{helmdeco}:

\begin{cor}
\label{poincarevec}
{\sf(Poincar\'e Estimate for Vector Fields)}
For all $v\in\HGcom$
$$\normLtom{v}\leq c_{p}\normLtom{\Grad v}.$$
\end{cor}

\begin{cor}
\label{poincaremaxten}
{\sf(Maxwell Estimate for Tensor Fields)}
The estimate
$$\normLtom{\T}\leq 
c_{m}\big(\normLtom{\Curl\T}^2+\normLtom{\Div\T}^2\big)^{1/2}$$
holds for all tensor fields $\T\in\HCcom\cap\HDom$.
\end{cor}

\begin{cor}
\label{helmdecoten}
{\sf(Helmholtz Decomposition for Tensor Fields)}
$$\Ltom=\Grad\HGcom\oplus\HDzom$$
\end{cor}

The last important tool is Korn's first inequality.

\begin{lem}
\label{korn}
{\sf(Korn's First Inequality)}
For all vector fields $v\in\HGcom$
$$\normLtom{\Grad v}\leq\sqrt{2}\normLtom{\sym\Grad v}.$$
\end{lem}

Here, we introduce the symmetric and skew-symmetric parts 
$$\sym\T:=\frac{1}{2}(\T+\trans{\T}),\quad
\skew\T:=\frac{1}{2}(\T-\trans{\T})$$
of a $(N\times N)$-tensor $\T=\sym\T+\skew\T$. 

\begin{rem}
\label{kornsqrttwo}
We note that the proof including the value of the constant is simple.
By density we may assume $v\in\Cicom$.
Twofold partial integration yields
$$\scpLtom{\pp_{n}v_{m}}{\pp_{m}v_{n}}
=\scpLtom{\pp_{m}v_{m}}{\pp_{n}v_{n}}$$
and hence
\begin{align*}
2\normLtom{\sym\Grad v}^2
&=\frac{1}{2}\sum_{n,m=1}^N\normLtom{\pp_{n}v_{m}+\pp_{m}v_{n}}^2\\
&=\sum_{n,m=1}^N\big(\normLtom{\pp_{n}v_{m}}^2
+\scpLtom{\pp_{n}v_{m}}{\pp_{m}v_{n}}\big)\\
&=\normLtom{\Grad v}^2+\normLtom{\div v}^2
\geq\normLtom{\Grad v}^2.
\end{align*}
More on Korn's first inequality can be found, e.g., in \cite{Neff00b}.
\end{rem}

\section{Results}

For tensor fields $\T\in\HCom$ we define the semi-norm
$$\dnorm{\T}:=\big(\normLtom{\sym\T}^2+\normLtom{\Curl\T}^2\big)^{1/2}.$$
The main step is to prove the following

\begin{lem}
\label{mainlem}
Let $\hat{c}:=\max\{2,\sqrt{5}c_{m}\}$. Then, for all $\T\in\HCcom$ 
$$\normLtom{\T}\leq\hat{c}\dnorm{\T}.$$
\end{lem}

\begin{proof}
Let $\T\in\HCcom$.
According to Corollary \ref{helmdecoten} we orthogonally decompose 
$$\T=\Grad v+\TS\in\Grad\HGcom\oplus\HDzom.$$
Then, $\Curl\T=\Curl\TS$ and we observe $\TS\in\HCcom\cap\HDzom$ since 
\begin{align}
\label{gradsubsetcurl}
\Grad\HGcom\subset\HCczom.
\end{align}
By Corollary \ref{poincaremaxten}, we have
\begin{align}
\label{estpsi}
\normLtom{\TS}\leq c_{m}\normLtom{\Curl\T}.
\end{align}
Then, by Lemma \ref{korn} and \eqref{estpsi} we obtain
\begin{align*}
\normLtom{\T}^2
&=\normLtom{\Grad v}^2+\normLtom{\TS}^2\\
&\leq2\normLtom{\sym\Grad v}^2+\normLtom{\TS}^2
\leq4\normLtom{\sym\T}^2+5\normLtom{\TS}^2,
\end{align*}
which completes the proof.
\end{proof}

The immediate consequence is our main result

\begin{theo}
\label{maintheo}
On $\HCcom$ the norms $\norm{\,\cdot\,}_{\HCom}$
and $\dnorm{\,\cdot\,}$ are equivalent.
In particular, $\dnorm{\,\cdot\,}$ is a norm on $\HCcom$
and there exists a positive constant $c$, such that
$$c\norm{\T}_{\HCom}^2\leq\dnorm{\T}^2
=\normLtom{\sym\T}^2+\normLtom{\Curl\T}^2$$
holds for all $\T\in\HCcom$.
\end{theo}

\begin{rem}
\label{maintheorem}
For a skew-symmetric tensor field $\T:\om\to\soN$ 
our estimate reduces to a Poincar\'e inequality in disguise, 
since $\Curl\T$ controls all partial derivatives of $\T$
(compare to \cite{Neff_curl06}) 
and the homogeneous tangential boundary condition 
for $\T$ is implied by $\T|_{\ga}=0$.
\end{rem}

Setting $\T:=\Grad v$ we obtain

\begin{rem}
\label{genkorn}
{\sf(Korn's First Inequality: Tangential-Variant)}
For all $v\in\HGcom$ 
\begin{align}
\label{kornineq}
\normLtom{\Grad v}\leq\hat{c}\normLtom{\sym\Grad v}
\end{align}
holds by Lemma \ref{mainlem} and \eqref{gradsubsetcurl}.
This is just Korn's first inequality from Lemma \ref{korn}
with a larger constant $\hat{c}$.
Since $\ga$ is connected, i.e., $\harmdio=\{0\}$, we even have
$$\Grad\HGcom=\HCczom.$$ 
Thus, \eqref{kornineq} holds for all
$v\in\HGom$ with $\Grad v\in\HCczom$, i.e.,
with $\Grad v_{n}$, $n=1,\dots,N$, normal at $\ga$,
which then extends Lemma \ref{korn} 
through the (apparently) weaker boundary condition.
\end{rem}

The elementary arguments above apply certainly 
to much more general situations, e.g.,
to not necessarily connected boundaries $\ga$
and to tangential boundary conditions
which are imposed only on parts of $\ga$.
These discussions are left to forthcoming papers.

\begin{acknow}
We thank the referee for pointing out a
missing assumption in a preliminary version of the paper.
\end{acknow}

\bibliographystyle{plain} 
\bibliography{/Users/paule/Library/texmf/tex/TeXinput/bibtex/paule,/Users/paule/Library/texmf/tex/TeXinput/bibtex/literatur1}

\vspace*{1cm}
\begin{tabular}{l}
\sf Patrizio Neff, Dirk Pauly, Karl-Josef Witsch\\
\\
Universit\"at Duisburg-Essen\\
Fakult\"at f\"ur Mathematik\\
Campus Essen\\
Universit\"atsstr. 2\\
45117 Essen\\
Germany\\
\\
\tt patrizio.neff@uni-due.de\\
\tt dirk.pauly@uni-due.de\\
\tt kj.witsch@uni-due.de
\end{tabular}

\end{document}